\documentclass[a4paper]{article}
\usepackage{amsmath, amsthm, amssymb}
\usepackage[latin1]{inputenc}
\usepackage[T1]{fontenc}
\usepackage[english]{babel}     
\usepackage{color}
\usepackage{tikz}                      
\usepackage{pgf}
\usepackage{xkeyval}

\usepackage{multirow}           
\usepackage{booktabs}
\usepackage{algorithm}
\usepackage{algorithmic}

\usepackage[dvips]{epsfig}
\usepackage{psfrag}
\usepackage[update,prepend]{epstopdf}

\DeclareMathOperator*{\argmin}{argmin}
\DeclareMathOperator*{\argmax}{argmax}
\newcommand*{\argminl}{\argmin\limits}
\newcommand*{\argmaxl}{\argmax\limits}

\newtheorem{Proposition}{Proposition}[section]

\newtheorem{Theorem}{Theorem}[section]
\newtheorem{Remark}{Remark}[section]

\newtheorem{Example}{Example}[section]
\numberwithin{equation}{section}
\newcommand\be{\begin{eqnarray*}}
\newcommand\ee{\end{eqnarray*}}
\newcommand\ben{\begin{eqnarray}}
\newcommand\een{\end{eqnarray}}
\newcommand{\comment}[1]{}

\def\Rd{{\mathbb R}^d}                               

\def\dvg{{\rm div}}                              

\def\Cspace{\mathsf{U}}                         
\def\Sspace{\mathsf{V}}                         
\def\LSspace{\mathsf{W}}                        
\def\Fspace{\mathsf{Q}}                         
\def\LFspace{\mathsf{H}}                        

\def\NT{\mid\!\mid\!\mid}                        
\def\NNN{\,{\pmb |\!\!\pmb |}\,}                 

\def\blow{\underline{c}}                        
\def\bup{\overline{c}}

\def\maj{\overline{M}}                          
\def\mij{\underline{M}}

\def\costup{\overline{J}}
\def\costlow{\underline{J}}

\def\err{{\rm err}}
\def\errlow{\underline{{\rm err}}}
\def\errup{\overline{{\rm err}}}

\def\H1o{H_0^1(\Omega)}   
\def\Hdiv{H(\dvg,\Omega)}                   

\title{Error estimates for a certain class of elliptic optimal control problems}
\author{O. Mali}

\begin{document}

\maketitle

\begin{abstract}
In this paper, error estimates are presented for a certain class of
optimal control problems with elliptic PDE-constraints. It is assumed that
in the cost functional the state is measured in terms of the energy norm
generated by the state equation. The functional a posteriori error estimates
developed by Repin in late 90's are applied to estimate the cost function value
from both sides without requiring the exact solution of the state equation. Moreover,
a lower bound for the minimal cost functional value is derived. A
meaningful error quantity coinciding with the gap between the cost functional
values of an arbitrary admissible control and the optimal control is introduced.
This error quantity can be  estimated  from both sides using the estimates for the
cost functional value. The theoretical results are confirmed by numerical tests.
\end{abstract}


\section{Introduction}
\label{se:int}

This paper presents two-sided estimates for the value of the cost functional (assuming
that the state equation can not be solved exactly) and shows how they can be used
 to generate estimates for a certain error quantity
 (cf. (\ref{eq:errdef}) and Theorem \ref{th:gen:mikh}).
In the case of unconstrained control, some estimates and numerical tests  have
been in presented in \cite{GaevskayaHoppeRepin2007}. In \cite{Repin2008},
the case of ``box constraints'' is treated.
Here, these results are extended considerably for constraints of more
general type, a new error quantity is introduced, and the results are confirmed
by numerical tests.

In section \ref{se:mod}, definitions and standard results related to optimal
control problems with elliptic state equation are recalled. Cost functionals
are assumed to be of a certain type, where the state is measured in terms of
the energy norm generated by the state equation. This is a special case of the
general theory which can be found, e.g., from monographs
\cite{Lions1971,Troltzsch2010}.

In section \ref{se:est}, the functional a posteriori error estimates (see
monographs \cite{NeittaanmakiRepin2004,Repin2008,MaliNeittaanmakiRepin2014}
and references therein) for the state equation are applied to generate two-sided
bounds for the
value of the cost functional. The strong connections between the estimates and
the principal relations generating the optimal control problem are underlined.
Theorem \ref{th:gen:mikh} (generalization of \cite[Ch. 9, Th. 9.14]{Repin2008} for
the case of constrained control) is the analog of the Mikhlin identity
(cf. Theorem \ref{th:gen:mikh}) for the optimal control problem. It
introduces a well motivated error quantity and shows how the estimates for the
cost function value can be used to generate two-sided bounds.

Some examples of optimal control problem of the type described in Sect.
\ref{se:mod} are discussed in Sect. \ref{se:ex}. Numerical tests in Sect.
\ref{se:num} depict how
the estimates can be combined with an arbitrary (conforming) numerical
method.

\section{Elliptic optimal control problem}
\label{se:mod}

\subsection{Definitions}

Let $\LSspace$, $\LFspace$, and $\Cspace$ be Hilbert spaces. Their
inner products and norms are denoted by subscripts, e.g., $(\cdot,\cdot)_{\LSspace}$
and $\| \cdot \|_{\LSspace}$.
Moreover, $\Sspace \subset \LSspace$ is a Hilbert space
generated by the inner product
$ (q,z)_{\Sspace} := (q,z)_{\LSspace} + (\Lambda q, \Lambda z)_{\LFspace}$, where
$\Lambda: \Sspace \rightarrow \LFspace$ is a linear,
bounded operator.
The injection from $\Sspace$ to $\LSspace$ is continuous and
$\Sspace$ is dense in $\LSspace$. Operator $\Lambda$ satisfies
a Friedrichs type inequality
\begin{equation} \label{eq:gen:Fri}
  \| q \|_{\LSspace} \leq c  \| \Lambda q \|_{\LFspace},
  \quad \forall \, q \in {\Sspace_0} ,
\end{equation}
where a subspace $\Sspace_0 \subset \Sspace$ is closed. Assume
$\Sspace_0 \subset \Sspace \subset \LSspace \subset \Sspace_0^*$,
where $\Sspace_0^*$ is the dual space of $\Sspace_0$.

Define linear bounded operators
$\mathcal B:\Cspace \rightarrow \Sspace_0^*$,
$\mathcal A:\LFspace \rightarrow \LFspace$,
$\mathcal N:\Cspace \rightarrow \Cspace$, where $\mathcal A$ and $\mathcal N$
are symmetric and positive definite,
\begin{equation*} 
    \blow \| q \|^2_{\LFspace}
    \leq ( \mathcal A q , q )_{\LFspace} \leq
    \bup \| q \|^2_{\LFspace}, \quad \forall \, q \in \LFspace
\end{equation*}
and
\begin{equation*} 
    \underline \kappa \| v \|^2_{\Cspace}
    \leq ( \mathcal N v , v )_{\Cspace} \leq
    \overline \kappa \| v \|^2_{\Cspace}, \quad \forall \, v \in \Cspace ,
\end{equation*}
where $\blow$ and $\bup$ ($\underline \kappa$ and $\overline \kappa$)
are positive constants. Thus, they generate inner products
\[
( q , z )_{\mathcal A} := (\mathcal A q, z)_\LFspace,
\quad
( q , z )_{\mathcal A^{-1}} := (\mathcal A^{-1} q, z)_\LFspace ,
\quad
( v , w )_{\mathcal N} := (\mathcal N v,w)_{\Cspace},
\]
and the respective norms
\[
\| q \|_{\mathcal A} := \sqrt{(\mathcal A q, q)_\LFspace} ,
\quad
\| q \|_{\mathcal A^{-1}} := \sqrt{ (\mathcal A^{-1} q, q)_\LFspace} ,
\quad
\| v \|_{\mathcal N} := \sqrt{ (\mathcal N v,v)_{\Cspace}} .
\]

The adjoint operators $\Lambda^*: \LFspace \rightarrow \Sspace_0^*$ and
$\mathcal B^*: \Sspace_0 \rightarrow \Cspace^*$ are defined
by the relations
\begin{equation*} 
\langle \Lambda^* z, q \rangle_{\Sspace_0} = ( z , \Lambda q )_{\LFspace} ,
\quad \forall \, z
\in \LFspace, \;  q \in \Sspace_0
\end{equation*}
and
\begin{equation} \label{eq:Badj}
\langle B v, q \rangle_{\Sspace_0} = \langle v, B^* q \rangle_{\Cspace} ,
\quad \forall \, v
\in \Cspace, \;  q \in \Sspace_0 ,
\end{equation}
where $\langle \cdot, \cdot \rangle_{\Sspace_0}$ denotes the pairing
of $\Sspace_0$ and its dual space $\Sspace_0^*$. By the Riesz
representation theorem, there exists an isomorphism (denoted, e.g., by
\mbox{$\mathcal I_\Cspace:\Cspace \rightarrow \Cspace^*$}) from any Hilbert space
onto the corresponding dual space.
The adjoint operator defines a subspace
\[
\Fspace := \{ q \in \LFspace \, | \, \Lambda^* q \in \LSspace \} \subset \LFspace .
\]
The norm to $\Sspace_0^*$ is
\begin{equation*} 
\NNN \ell \NNN := \sup\limits_{q \in \Sspace_0 \atop q \neq 0}
\frac{|\langle \ell, q \rangle_{\Sspace_0}|}{ \| \Lambda q \|_{\mathcal A}} .
\end{equation*}

Consider a bilinear form $a: \Sspace_0 \times \Sspace_0 \rightarrow \mathbb{R}$,
\begin{equation*} 
a(q,z) := ( \mathcal A \Lambda q , \Lambda z )_{\LFspace} .
\end{equation*}
It is $\Sspace_0$ -elliptic and continuous and generates
an energy norm $ \NT q \NT := \sqrt{a(q,q)} $
in $\Sspace_0$.

\subsection{Optimal control problem}

The state equation is
\begin{equation}
\label{eq:state:weak}
a(y(v),q) = \langle \ell + \mathcal Bv, q \rangle_{\Sspace_0} ,
\quad \forall q \in \Sspace_0 ,
\end{equation}
where $\ell \in \Sspace_0^*$, $v \in \Cspace_{\rm ad} \subset \Cspace$ is the control,
and $y(v) \in \Sspace_0$ is the corresponding state. Let
$\Cspace_{\rm ad} \subset \Cspace$ be a non-empty,
convex, and closed set.
The cost functional $J:\Cspace \rightarrow \mathbb{R}$ is
\begin{equation}
\label{eq:cost} J(v) := \NT y(v) - y^d \NT^2
+ \| v-u^d \|_{\mathcal N}^2 ,
\end{equation}
where $u^d \in \Cspace$ and $y^d \in \Sspace_0$.
The optimal control problem is to find $u \in \Cspace_{\rm ad}$,
such that
\begin{equation} \label{eq:ocp}
J(u) \leq J(v) , \quad \forall v \in \Cspace_{\rm ad} .
\end{equation}
Under earlier assumptions, $J$ is $\Cspace$-elliptic, coercive, and lower
semi-continuous. Thus, the solution of the optimal control problem exists and is unique
(see, e.g., \cite[Chap. II, Th. 1.2]{Lions1971}).

\begin{Remark}
Cost functional of type
\[
J_2(v) := \| \Lambda y(v) - \sigma^d \|^2_{\mathcal A}
+ \| v-u^d \|_{\mathcal N}^2
\]
can be shifted using a projection: Find $y^d \in \Sspace_0$ such that
\[
(\mathcal A (\Lambda y^d - \sigma^d) , \Lambda q )_{\LSspace} = 0 , \quad \forall q \in \Sspace_0 .
\]
Then,
$J(v) = J_2(v) - \| \Lambda y^d - \sigma^d \|^2_{\mathcal A}$
\end{Remark}

The derivative of $J$ at $v$ is
\begin{multline} \label{eq:Jderiv}
\langle J'(v) , w \rangle_\Cspace =
\lim\limits_{t \rightarrow 0_+} \tfrac{1}{t}\left( J(v+tw)-J(v) \right)
=
2 \langle \mathcal B w , y(v) - y^d \rangle_{\Sspace_0} + ( v-u^d , w )_{\mathcal N} \\
=
2 ( \mathcal I_\Cspace^{-1} \mathcal B^*(y(v)-y^d) + \mathcal N (v-u^d) , w )_\Cspace .
\end{multline}
The necessary conditions for the optimal control
problem (\ref{eq:ocp}) are
(\ref{eq:state:weak}) and
\begin{equation} \label{eq:nec:der}
\langle J'(u) , v-u \rangle_\Cspace \geq 0, \quad \forall v \in \Cspace_{\rm ad}
\end{equation}
(see, e.g., \cite[Ch. I, Th. 1.3]{Lions1971}, \cite[Le. 2.21]{Troltzsch2010}), i.e.,
\begin{equation} \label{eq:nec:const}
( \mathcal I_\Cspace^{-1} \mathcal B^*(y(u)-y^d) + \mathcal N (u-u^d) , v-u )_\Cspace
\geq 0,
\quad \forall v \in \Cspace_{\rm ad} .
\end{equation}
Note that for the cost functional of type (\ref{eq:cost}), there is
no need to define an adjoint state to present the necessary conditions
(compare \cite[Chap. II, Th. 1.4]{Lions1971}).

The following proposition
(dating back to 
\cite{Moreau1965}, see, e.g.,
\cite[Chap. I, Pr. 2.2]{EkelandTemam1976} or \cite[Chap. 7, Pr. 7.4]{Clarke2013})
allows
to write (\ref{eq:nec:const}) in a different form.
\begin{Proposition} \label{pr:proj}
Including the earlier assumptions, let $x \in \Cspace$.
Then, the following conditions are equivalent,
\[
\begin{tabular}{ll}
(i) $\quad \quad$&
$
(u-x,v-u)_{\mathcal N} \geq 0, \quad \forall v \in \Cspace_{\rm ad} ,
$
\\
(ii) &
$
\| x-u \|_{\mathcal N} = \inf\limits_{v \in \Cspace_{\rm ad}} \| x- v  \|_{\mathcal N} ,
$
\\
(iii) &
$u = \Pi_{\rm ad}^{\mathcal N} x$, where
$\Pi_{\rm ad}^{\mathcal N}: \Cspace \rightarrow \Cspace_{\rm ad}$ is a projection.
\end{tabular}
\]
\end{Proposition}
\begin{proof}
Assume (i). The identity
\[
\| x - v \|_{\mathcal N}^2 - \| x - u \|_{\mathcal N}^2 =
\| u - v \|_{\mathcal N}^2 + 2 (u-x,v-u)_{\mathcal N} \geq 0
\]
leads at $\| x - u \|_{\mathcal N} \leq \| v - x \|_{\mathcal N}$
for arbitrary $v \in \Cspace_{\rm ad}$,
i.e., (ii).

Assume (ii). Let $v \in \Cspace_{\rm ad}$ be arbitrary
and $t \in (0,1)$, then by the convexity of $\Cspace_{\rm ad}$
\[
\| x - u \|_{\mathcal N}^2 \leq \| x - ((1-t)u + tv) \|_{\mathcal N}^2
= \| (x - u ) + t(u - v) \|_{\mathcal N}^2 .
\]
Expanding the right side leads at
$2 t (x-u,u-v)_{\mathcal N} \leq t^2 \| t(u - v) \|_{\mathcal N}^2$,
tending $t$ to zero yields (i).

Conditions (ii) and (iii) equal by definition.
\end{proof}
Proposition \ref{pr:proj} and (\ref{eq:nec:const}) yield the so called projection condition
\begin{equation} \label{eq:nec:proj}
u = \Pi_{\rm ad}^{\mathcal N}
\left( u^d - \mathcal N^{-1} \mathcal I_\Cspace^{-1} \mathcal B^*(y(u)-y^d) \right) .
\end{equation}

\begin{Remark} \label{re:Nid}
Typical choice is $\mathcal N = \alpha {\rm Id}$, where $\alpha > 0$ and ${\rm Id}$ denotes
the identity mapping.
Then (\ref{eq:nec:proj}) becomes
\[
u = \Pi_{\rm ad}
\left( u^d - \tfrac{1}{\alpha} \mathcal I_\Cspace^{-1} \mathcal B^*(y(u)-y^d) \right) .
\]
\end{Remark}

\begin{Remark} \label{re:nec:linear}
If $\Cspace_{\rm ad} = \Cspace$, then $\Pi_{\rm ad}^{\mathcal N} = {\rm Id}$ and
(\ref{eq:nec:proj}) reduces to
\begin{equation} \label{eq:nec:unconst}
u = u^d - \mathcal N^{-1} \mathcal I_\Cspace^{-1} \mathcal B^*(y(u)-y^d) .
\end{equation}
Substituting (\ref{eq:nec:unconst}) to (\ref{eq:state:weak})
yields a following
linear problem: Find $y(u) \in \Sspace_0$ satisfying
\begin{multline} \label{eq:gen:linear}
a(y(u),z) +
\langle \mathcal B \mathcal N^{-1} \mathcal I_\Cspace^{-1} \mathcal B^* y(u) , z \rangle_{\Sspace_0} \\
= \langle \ell + \mathcal B u^d ,z \rangle_{\Sspace_0} +
\langle \mathcal B \mathcal N^{-1} \mathcal I_\Cspace^{-1} \mathcal B^*y^d  , z \rangle_{\Sspace_0}
\quad \forall z \in \Sspace_0 .
\end{multline}
\end{Remark}

\section{Estimates}
\label{se:est}

\subsection{Estimates for the state equation}

The solution $y(v) \in \Sspace_0$ of
(\ref{eq:state:weak}) minimizes a quadratic
energy functional (see, e.g., \cite[Chapter I, Theorem 1.2 and
Remark 1.5]{Lions1971} ), i.e.,
\begin{equation} \label{eq:state:var}
E(y(v)) \leq E( q )
:= \NT q \NT^2  - 2 \langle \ell + \mathcal B v, q \rangle_{\Sspace_0} ,
\quad \forall q \in \Sspace_0 .
\end{equation}
The benefit for measuring $y(v)-y^d$ in the
$\NT \cdot \NT$-norm in (\ref{eq:cost})
(instead of, e.g., $\| \cdot \|_{\LSspace}$-norm) is due to the following
results (Theorem \ref{th:Mikh} is due to \cite{Mikhlin1964} and generalized in
\cite{Repin2008}).
\begin{Theorem} \label{th:Mikh}
Let $y(v)$ be the solution of
(\ref{eq:state:var}) and $z \in \Sspace_0$ be
arbitrary, then
\begin{equation} \label{eq:Mikh}
\NT y(v) - z \NT^2
= E(z) - E(y(v)) .
\end{equation}
\end{Theorem}
\begin{proof}
By (\ref{eq:state:weak}),
\begin{align*}
\NT y(v) - z \NT^2  & =
\NT y(v) \NT^2 - 2 a( y(v) , z )
+ \NT z \NT^2 \\
& \quad - 2 \left( a( y(v), y(v)) + \langle \ell+\mathcal B v,y(v)\rangle_{\Sspace_0} \right) \\
& =
- \NT y(v) \NT^2 + 2 \langle \ell + \mathcal B v,z \rangle_{\Sspace_0}
+ \NT z \NT^2 - 2 \langle \ell + \mathcal B v,y(v) \rangle_{\Sspace_0} \\
& =
E (z) - E (y(v)) .
\end{align*}
\end{proof}
\begin{Theorem} \label{th:state:est}
Let $y(v)$ be the solution of (\ref{eq:state:var}) and
$z \in \Sspace_0$ be arbitrary, then
\begin{equation*} 
\sup\limits_{q \in \Sspace_0} \mij^2(z,q,v)
= \NT y(v)- z \NT^2 =
\inf\limits_{\tau \in \Fspace \atop \beta > 0} \maj^2(z,\tau,\beta,v) ,
\end{equation*}
where
\begin{equation}  \label{eq:minor}
\mij^2(z,q,v) := E(z) - E(q)
\end{equation}
and
\begin{equation}  \label{eq:major}
\maj^2(z,\tau,\beta,v) := (1+\beta) \| \tau - \mathcal A \Lambda z \|_{\mathcal A^{-1}}^2 +
\frac{1+\beta}{\beta} \NNN \Lambda^* \tau + \mathcal B v + \ell \NNN^2 .
\end{equation}
\end{Theorem}
\begin{proof}
$\mij^2$ is obtained directly from (\ref{eq:state:var}) and (\ref{eq:Mikh}). For $\maj^2$,
see, e.g., \cite[Chap. 6, (6.2.3)]{NeittaanmakiRepin2004},
\cite[Chap. 7, (7.1.19)]{Repin2008}. Upper bounds of more general
type have been presented already in \cite{Repin1997,Repin2000}.
\end{proof}

\begin{Remark} \label{Re:winfyinf}
It is easy to confirm that the supremum over $\mij^2$ is obtained at $q = y(v)$ and
 the infimum over $\maj^2$ is attained
 at $\tau = \mathcal A \Lambda y(v)$ and $\beta \rightarrow 0$.
\end{Remark}

\subsection{Estimates for the cost functional}

Applying Theorem \ref{th:state:est} to the first term of (\ref{eq:cost}), leads to
two-sided bounds for $J(v)$. These bounds are guaranteed, have no gap,
 and do not depend on $y(v)$, i.e., they do not require the solution of the state equation.
\begin{Theorem} \label{th:costest}
For any $v \in \Cspace$,
\begin{equation} \label{eq:Jbounds}
\sup\limits_{q \in \Sspace_0} \costlow(v,q)
=
J(v)
=
\inf\limits_{\tau \in \Fspace \atop \beta>0} \costup(v,\tau,\beta),
\end{equation}
where
\begin{equation} \label{eq:defcostlow}
\costlow(v,q) := \mij^2(y^d,q,v) + \| v - u^d \|_{\mathcal N}^2
\end{equation}
and
\begin{equation} \label{eq:defcostup}
\costup(v,\tau,\beta) := \maj^2(y^d,\tau,\beta,v)
+ \| v - u^d \|_{\mathcal N}^2 .
\end{equation}
\end{Theorem}
Theorem \ref{th:costest} can be used to estimate $J(u)$. By (\ref{eq:ocp}) and
(\ref{eq:Jbounds}),
\begin{equation} \label{eq:Jubounds}
\inf\limits_{v \in \Cspace_{\rm ad}} \costlow(v,q)
\leq
J(u)
\leq
\costup(v,\tau,\beta), \quad
\forall \, q \in \Sspace_0, \, v \in \Cspace, \, \tau \in \LFspace, \; \beta > 0 ,
\end{equation}
where all inequalities hold as equalities if $v=u$, $q=y(u)$,
$\tau = \mathcal A \Lambda y(u)$, and $\beta \rightarrow 0$.
In view of
(\ref{eq:Jubounds}), it is very important that
the minimizer of $\costlow(v,q)$ over $v \in \mathcal \Cspace_{\rm ad}$ can be
explicitly computed.
Computation of the minimizers of $\costup$ require further assumptions of the structure
of the problem (cf. Propositions \ref{pr:ex1:Jupmin} and \ref{pr:ex2:Jupmin}).
\begin{Proposition} \label{pr:Jlowmin}
For all $v \in \Cspace_{\rm ad}$ and $q \in \Sspace_0$,
\begin{align}
\label{eq:Jlow:vmin}
\costlow(\hat v(q), q) & = \inf\limits_{v \in \Cspace_{\rm ad}} \costlow(v,q), \\
\nonumber 
\costlow(v, \hat q(v)) & = \sup\limits_{q \in \Sspace_0} \costlow(v,q),
\end{align}
where $\hat q(v) = y(v)$ (from (\ref{eq:state:weak})) and
\begin{equation} \label{eq:vhatdef}
\hat v (q) := \Pi_{\rm ad}^{\mathcal N}
\left( u^d + \mathcal N^{-1} \mathcal B^*(y^d-q) \right) .
\end{equation}
\end{Proposition}
\begin{proof}
The condition $\hat q(v) = y(v)$ follows directly from Remark \ref{Re:winfyinf}.

By (\ref{eq:state:var}), (\ref{eq:minor}), and (\ref{eq:defcostlow}),
$\costlow$ has the following form
\begin{multline*}
\costlow(v,q) =
\NT y^d \NT^2 - 2 \langle \ell,y^d \rangle
- \NT q \NT^2 + 2 \langle \ell,q \rangle + 2\langle B v,q-y^d \rangle_{\Sspace_0}
+ \| v - u^d \|_{\mathcal N}^2 \\
= \| v \|_{\mathcal N}^2 - 2 (v , u^d)_{\mathcal N} - 2 \langle \mathcal B v,y^d \rangle_{\Sspace_0}
- \NT q \NT^2 + 2 \langle \ell,q \rangle + 2 \langle \mathcal B v,q \rangle_{\Sspace_0} + {\rm const.}
\end{multline*}
Clearly, it is quadratic w.r.t
$v$ and the minimizer $\hat v \in \Cspace_{\rm ad}$ is identified
by the following variational inequality
(see, e.g., \cite[Chap. I, Th. 1.2]{Lions1971}
):
\[
(\hat v,v-\hat v)_{\mathcal N} \geq (v-\hat v,u^d)_{\mathcal N}
+\langle {\mathcal B} (v-\hat v) , y^d - q \rangle_{\Sspace_0},
\quad \forall v \in \mathcal \Cspace_{\rm ad} .
\]
Reorganizing and (\ref{eq:Badj}) yields
\begin{equation*} 
\left( \hat v-u^d+ \mathcal N^{-1} {\mathcal B}^*(q-y^d), v - \hat v \right )_{\mathcal N} \geq 0 ,
\quad \forall v \in \mathcal \Cspace_{\rm ad} ,
\end{equation*}
and Proposition \ref{pr:proj} leads at (\ref{eq:vhatdef}).
\end{proof}
\begin{Remark} \label{re:boundinfo}
By (\ref{eq:Jbounds}) and (\ref{eq:Jubounds}), $\costup(v,\tau,\beta)$ is an upper
bound of $J(u)$ for all $v \in \Cspace_{ad}$, $\tau \in \Fspace$, and $\beta>0$ and
$\costlow(v,q)$ is a lower bound for $J(v)$ for all $q \in \Cspace_{ad}$,
but it is a lower bound of $J(u)$ only if $v=\hat v(q)$ (see (\ref{eq:vhatdef})).
\end{Remark}
\begin{Remark} \label{re:costlowreform}
Lower bound $\costlow$ generates a saddle point formulation for the original
optimal control problem (\ref{eq:ocp}). Find
$(\tilde v, \tilde q)$ satisfying
\begin{equation} \label{eq:saddle:ineg}
\costlow(\tilde v,q) \leq \costlow(\tilde v,\tilde q) \leq \costlow(v,\tilde q),
\quad \forall v \in \Cspace_{\rm ad} , q \in \Sspace_0 .
\end{equation}
Note that $\costlow$ is convex,
lower semi-continuous,
and coercive w.r.t. $v$ and concave,
upper semi-continuous, and anti-coercive
w.r.t $q$, $\Cspace_{\rm ad}$ is convex, closed, and non-empty, and
$\Sspace_0$ is convex, closed, and non-empty. Thus,
the solution of (\ref{eq:saddle:ineg})
exists and is unique (see, e.g.,  \cite[Chap. VI, Pr. 2.4]{EkelandTemam1976}).
By Remark \ref{Re:winfyinf}, $\tilde v = u$ and $\tilde q = y(u)$. Moreover,
$\hat v(y(u)) = u$, where $\hat v$ is defined in (\ref{eq:vhatdef}).
The left and right-hand-side of (\ref{eq:saddle:ineg})
yield (\ref{eq:state:var}) and (\ref{eq:nec:const}) (i.e., necessary conditions
(\ref{eq:state:weak}) and (\ref{eq:nec:der})), respectively.
\end{Remark}
\begin{Remark} \label{re:costupmin}
By (\ref{eq:Jubounds}),
$J(u) \leq J(v) \leq \costup(v,\tau,\beta)$ and it is easy to see that
$J(u)=\lim\limits_{\beta \rightarrow 0} \costup(u,\mathcal A\Lambda y(u),\beta)$.
Thus, the upper bound generates a minimization problem
\[
\costup(u,\mathcal A \Lambda y(u),0) =
\min\limits_{v \in \Cspace_{\rm ad}, \tau \in \Fspace \atop \beta > 0}
\costup(v,\tau,\beta) ,
\]
where the constraint related to (\ref{eq:state:weak}) does not appear.
\end{Remark}

\subsection{Estimates for an error quantity}

The following identity can be viewed as an analog of (\ref{eq:state:var})
for the optimal control problem.
\begin{Theorem} \label{th:gen:mikh}
For any $v \in \Cspace_{\rm ad}$,
\begin{equation} \label{eq:cont:equi}
\NT y(v) - y(u) \NT^2 + \| v - u \|_{\mathcal N}^2
+ \langle J'(u) , v-u \rangle_\Cspace = J(v) - J(u)  .
\end{equation}
\end{Theorem}
\begin{proof}
We have,
\begin{multline*}
J(v)-J(u)
=
\NT y(v) - y(u) \NT^2 + 2a(y(v)-y(u),y(u)-y^d) \\
+ \| v - u \|_{\mathcal N}^2 + 2( v-u, u- u^d )_{\mathcal N} .
\end{multline*}
By (\ref{eq:state:weak}) and (\ref{eq:Jderiv}),
\[
a(y(v)-y(u),y(u)-y^d) = \langle \mathcal B (v-u), y(u)-y^d \rangle_{\Sspace_0}
\]
and
\[
2 a(y(v)-y(u),y(u)-y^d)+ 2 ( v-u, u- u^d )_{\mathcal N}
= \langle J'(u) , v-u \rangle_\Cspace.
\]
\end{proof}
\begin{Remark}
If $\Cspace_{\rm ad} = \Cspace$, then $\langle J'(u) , v \rangle_\Cspace = 0$, for all $v \in \Cspace$ and
(\ref{eq:cont:equi}) reduces to \cite[Ch. 9, Th. 9.14]{Repin2008}.
\end{Remark}

Equality (\ref{eq:cont:equi}) shows that it is reasonable to include $\langle J'(u) , v-u \rangle_\Cspace$
to the applied error measure.
Obviously, $\langle J'(u) , v-u \rangle_\Cspace$ is positive for any
$v \in \Cspace_{\rm ad}$
by (\ref{eq:nec:der}), it is convex and vanishes
if $v=u$. Thus, the error measure is
\begin{equation} \label{eq:errdef}
\err^2( v ) :=
\NT y(v) - y(u) \NT^2 + \| v - u \|_{\mathcal N}^2 + \langle J'(u) , v-u \rangle_\Cspace .
\end{equation}
The ``derivative weight'' guarantees that the sensitivity of the cost
functional at the optimal control is taken into account. Most importantly, $\err(v)$
can be estimated from both sides by computable functionals, which do not require
the knowledge of the optimal control $u$, the respective state $y(u)$, or the exact
state $y(v)$. Indeed, applying
(\ref{eq:Jbounds}), (\ref{eq:Jubounds}), and (\ref{eq:Jlow:vmin}) to
the right hand side of (\ref{eq:cont:equi}) yields the following theorem:
\begin{Theorem} \label{th:cont:est}
For any $v \in \Cspace_{\rm ad}$,
\begin{equation} \label{eq:cont:est}
\sup\limits_{q \in \Sspace_0, v_2 \in \Cspace_{\rm ad}, \atop \tau \in \Fspace, \beta > 0}
\errlow^2(v,q,v_2,\tau,\beta)
= \err^2( v ) =
\inf\limits_{\tau \in \Fspace, \beta>0, \atop q_2 \in \Sspace_0}
\errup^2(v,\tau_2,\beta_2,q_2),
\end{equation}
where
\begin{equation*} 
\errlow^2(v,q,v_2,\tau,\beta) :=
\costlow(v,q) - \costup(v_2,\tau,\beta)
\end{equation*}
and
\begin{equation*} 
\errup^2(v,\tau_2,\beta_2,q_2) :=
\costup(v,\tau_2,\beta_2) - \costlow(\hat v(q_2),q_2) .
\end{equation*}
\end{Theorem}
\begin{Remark}
By Remark \ref{re:boundinfo}, (\ref{eq:defcostlow}), (\ref{eq:defcostup}), and
(\ref{eq:cont:equi}), the equality (\ref{eq:cont:est}) is attained at
\begin{equation*} 
\errlow^2(v,y(v),u,\mathcal A \Lambda y(u),0)
= \err^2( v ) =
\errup^2(v,\mathcal A \Lambda y(v),0,y(u)) .
\end{equation*}
\end{Remark}
\begin{Remark}
Obviously $J(v)$ and $\err^2(v)$ are positive. However, e.g., the lower bound
$\costlow(\hat v(q_2),q_2)$ for
$J(u)$ may be negative if $q_2$ is not close enough to $y(u)$ and
$\errlow^2(v,q,v_2,\tau,\beta)$ may be negative value if $v_2$ is not ``good
enough'' in comparison with $v$, or the upper bound $\costup(v_2,\tau,\beta)$
is not ``sharp enough''.
\end{Remark}

\section{Examples, algorithms and numerical tests}
\label{se:examples}

\subsection{Examples}
\label{se:ex}

In the following examples,
the domain $\Omega \subset \mathbb{R}^d$ is open, simply connected
and has a piecewise Lipschitz-continuous boundary $\Gamma$.
Spaces are
$\LSspace=L^2(\Omega)$, \mbox{$\Sspace = H^1(\Omega)$},
\mbox{$\LFspace = L^2(\Omega,\mathbb{R}^d)$}, and \mbox{$\Fspace = \Hdiv$}.
Operators are
\mbox{$\Lambda = \nabla$}, \linebreak \mbox{$\Lambda^* = -\dvg$},
\mbox{$\mathcal A = {\rm Id}$},
and $N=\alpha {\rm Id}$ ($\alpha>0$). Then
$a(q,z):=(\nabla q, \nabla z)_{L^2(\Omega,\Rd)}$ and \linebreak
\mbox{$\NT w \NT = \| \nabla w \|_{L^2(\Omega,\Rd)}$}.
The examples differ only by the selection of $\Sspace_0$,
$\Cspace$, $\mathcal B$, and $\ell$.

\subsubsection{Dirichlet problem, distributed control}

Let $\Cspace:= L^2(\Omega)$,
$\Sspace_0 := H_0^1(\Omega)$, and
$\langle \ell,w \rangle=(f,w)_{L^2(\Omega)}$,
where $f \in L^2(\Omega)$.
Moreover, $B={\rm Id}$, i.e.,
$\langle Bv, q  \rangle = ( v, q)_{L^2(\Omega)} $.
The analog of (\ref{eq:gen:Fri}) is the Friedrichs
inequality
\[
\| q \|_{L^2(\Omega)} \leq c_\Omega \| \nabla q \|_{L^2(\Omega,\mathbb{R}^d)} ,
\quad \forall q \in \H1o .
\]

The cost functional (\ref{eq:cost}) is
\begin{equation}
\label{eq:ex1:cost}
J(v) := \| \nabla (y(v) - y^d) \|_{L^2(\Omega,\Rd)}^2
+ \alpha \| v - u^d \|_{L^2(\Omega)}^2 .
\end{equation}
The state equation (\ref{eq:state:weak}) is
\begin{equation} \label{eq:ex1:state}
(\nabla y(v), \nabla z)_{L^2(\Omega,\Rd)} =
( f+v, z)_{L^2(\Omega)}, \quad \forall z \in \H1o
\end{equation}
and it has the classical form
\begin{equation*}
\left\{
\begin{array}{rclr}
- \Delta y(v) & = & f + v \quad & \textrm{a.e. in } \Omega, \\
             y(v) & = & 0 \quad & \textrm{on } \Gamma .
\end{array}
\right.
\end{equation*}
The majorant (\ref{eq:major}) is
\begin{equation*} 
\maj^2(q,\tau,\beta,v) = (1+\beta) \| \tau - \nabla z \|_{L^2(\Omega,\mathbb{R}^d)}^2 +
\frac{1+\beta}{\beta} c_\Omega^2 \| \dvg \tau + f + v \|_{L^2(\Omega)}^2 .
\end{equation*}
The counterpart of the Proposition \ref{pr:Jlowmin} is below.
\begin{Proposition} \label{pr:ex1:Jupmin}
For all $v \in \Cspace_{\rm ad}$, $\tau \in \Hdiv$, and $\beta > 0$
\begin{align*}
\costup(\hat v(\tau,\beta),\tau,\beta) & = \inf\limits_{v \in \Cspace_{\rm ad}} \costup(v,\tau,\beta) , \\
\costup(v,\hat \tau(v,\beta),\beta) & = \inf\limits_{\tau \in \Hdiv} \costup(v,\tau,\beta) , \\
\costup(v,\tau,\hat \beta(v,\tau)) & = \inf\limits_{\beta > 0} \costup(v,\tau,\beta) ,
\end{align*}
where
\begin{equation} \label{eq:ex1:hatvdef}
\hat v(\tau,\beta) = \Pi_{\rm ad} \left(
\tfrac{\alpha \beta}{(1+\beta) c_\Omega^2} u^d
    - \dvg \tau - f
\right) ,
\end{equation}
$\hat \tau := \hat \tau (v,\beta)$ satisfies
\begin{multline} \label{eq:ex1:hattaudef}
\beta (\hat \tau, \xi)_{L^2(\Omega,\mathbb{R}^d)} +
c_\Omega^2 (\dvg \hat \tau, \dvg \xi )_{L^2(\Omega)} \\
=
\beta (\nabla y^d, \xi )_{L^2(\Omega,\mathbb{R}^d)} +
c_\Omega^2 (f+v, \dvg \xi )_{L^2(\Omega)},
\quad \forall \xi \in \Hdiv ,
\end{multline}
and
\begin{equation} \label{eq:ex1:hatbetadef}
\hat \beta(v,\tau) = \frac{c_\Omega \| \dvg \tau + f + v \|_{L^2(\Omega)}}
{\| \tau - \nabla y^d \|_{L^2(\Omega,\mathbb{R}^d)}} .
\end{equation}
\end{Proposition}
\begin{proof}
The upper bound $\costup$ can be rewritten as follows,
\begin{multline*}
\costup(v,\tau,\beta) =
(1+\beta) \| \tau - \nabla z \|_{L^2(\Omega,\mathbb{R}^d)}^2 +
\tfrac{1+\beta}{\beta} c_\Omega^2 \| \dvg \tau + f + v \|_{L^2(\Omega)}^2 +
\alpha \| v -u^d \|^2_{L^2(\Omega)} \\
=
\left( \tfrac{1+\beta}{\beta} c_\Omega^2 + \alpha \right)
\| v \|_{L^2(\Omega)}^2
- 2 \left( \alpha u^d - \tfrac{1+\beta}{\beta} c_\Omega^2 (\dvg \tau + f), v \right)_{L^2(\Omega)}
+ \textrm{ const w.r.t } v .
\end{multline*}
Thus, the minimizer $\hat v \in \Cspace_{\rm ad}$ satisfies
\[
\left( \tfrac{1+\beta}{\beta} c_\Omega^2 + \alpha \right)
( \hat v , w - \hat v )_{L^2(\Omega)}
\geq
\left( \alpha u^d - \tfrac{1+\beta}{\beta} c_\Omega^2 (\dvg \tau + f), w - \hat v \right)_{L^2(\Omega)},
\; \forall w \in \Cspace_{\rm ad} .
\]
Reorganizing leads at
\[
( \hat v - \tfrac{\alpha \beta}{(1+\beta) c_\Omega^2} u^d
    + \dvg \tau + f, w - \hat v )_{L^2(\Omega)} ,
\quad \forall w \in \Cspace_{\rm ad} ,
\]
and Proposition \ref{pr:proj} yields (\ref{eq:ex1:hatvdef}).

Condition (\ref{eq:ex1:hattaudef}) can be
easily derived, since $\maj^2$ is quadratic w.r.t. \linebreak
\mbox{$\tau \in \Hdiv$} and
(\ref{eq:ex1:hatbetadef}) results from solving a one-dimensional minimization
problem.
\end{proof}
The relation (\ref{eq:vhatdef}) becomes
\begin{equation} \label{eq:ex1:vminstep}
\hat v (q) = \Pi_{\rm ad} \left( u^d + \tfrac{1}{\alpha} (y^d-q) \right) ,
\end{equation}
where $\Pi_{\rm ad}:L^2(\Omega) \rightarrow \Cspace_{\rm ad}$ is
a projection.

\begin{Example} \label{re:ex1:linear}
If $\Cspace_{\rm ad} = L^2(\Omega)$, then by
(\ref{eq:gen:linear}) $y(u) \in H_0^1(\Omega)$ satisfies
\begin{multline} \label{eq:ex1:linear}
(\nabla y(u) , \nabla z)_{L^2(\Omega,\Rd)}
+ \tfrac{1}{\alpha} (y(u),z)_{L^2(\Omega)} \\
=
( f + \tfrac{1}{\alpha} y^d + u^d,z )_{L^2(\Omega)}, \quad
\forall z \in \H1o
\end{multline}
and $\Pi_{\rm ad} = {\rm Id}$ in (\ref{eq:ex1:hatvdef}) and (\ref{eq:ex1:vminstep}).
\end{Example}

\begin{Example} \label{re:ex1:Uad:loc}
Let
\begin{equation} \label{eq:ex1:Uad:loc}
\Cspace_{\rm ad} = \{ v \in L^2(\Omega) \, | \,
\psi_{-} \leq v \leq \psi_{+} \quad \textrm{a.e in } \Omega \} ,
\end{equation}
then the projection operator
$\Pi_{\rm ad}:L^2(\Omega) \rightarrow \Cspace_{\rm ad}$
is
\[
\Pi_{\rm ad} v = \min \left\{ \psi_{+} , \max \left\{ \psi_{-} , v \right\} \right\} .
\]
\end{Example}

\begin{Example} \label{re:ex1:Uad:glo}
Let
\begin{equation*} 
\Cspace_{\rm ad} = \{ v \in L^2(\Omega) \, | \,
\| v \|_{L^2(\Omega)} \leq M \} ,
\end{equation*}
then the projection operator
$\Pi_{\rm ad}:L^2(\Omega) \rightarrow \Cspace_{\rm ad}$
is
\[
\Pi_{\rm ad} v = \left\{ \begin{array}{ll}
\frac{M v}{\| v \|_{L^2(\Omega)}} \quad & \textrm{ if } \| v \|_{L^2(\Omega)} > M , \\
v  \quad & \textrm{ else }
\end{array} \right.
\]
\end{Example}

Finally, functional a posteriori error estimates for the problem
(\ref{eq:ex1:linear}) are
recalled.  (see, e.g.,
\cite[Ch. 4.2]{Repin2008}, and \cite[Ch. 3.2]{MaliNeittaanmakiRepin2014}).

\begin{Theorem}
Let $y$ be the solution of (\ref{eq:ex1:linear}) and $z \in \H1o$, then
\[
\| \nabla(y-z) \|_{L^2(\Omega,\mathbb{R}^d)}^2 +
\tfrac{1}{\alpha} \| y-q \|_{L^2(\Omega)}^2 =
\inf\limits_{\tau \in \Hdiv, \beta > 0, \atop \nu \in L^2(\Omega,[0,1])}
\maj (z,\tau,\beta,\nu) ,
\]
where
\begin{multline*}
\maj (z,\tau,\beta,\nu) :=
(1+\beta) \| \nabla z - \tau \|_{L^2(\Omega,\mathbb{R}^d)}^2 \\
+ \tfrac{1+\beta}{\beta} c_\Omega^2 \| \nu \mathcal R(z,\tau) \|_{L^2(\Omega)}^2 +
\alpha \| (1-\nu) \mathcal R(z,\tau) \|_{L^2(\Omega)}^2
\end{multline*}
and
\[
\mathcal R(z,\tau) = \dvg \tau - \tfrac{1}{\alpha} z + f + \tfrac{1}{\alpha} y^d + u^d .
\]
\end{Theorem}

\subsubsection{Neumann problem, boundary control}

The boundary $\Gamma$ consists of two parts $\Gamma_{N} \cup \Gamma_D$, where
$\Gamma_D$ has a positive measure. By the trace theorem
there exists a bounded linear mapping \mbox{$\gamma: \H1o \rightarrow L^2(\Gamma_{N})$},
\[
\| \gamma q \|_{L^2(\Gamma_{N})} \leq c \| q \|_{H^1(\Omega)} ,
\]
such that $\gamma v = v_{| \Gamma}$ for all $v \in C^1(\bar \Omega)$. 
Let $\Cspace:= L^2(\Gamma_{N})$ and
\[
\Sspace_0 := V_0 :=
\{ w \in H^1(\bar \Omega) \, | \, w \textrm{ has zero trace on } \Gamma_D \}.
\]
Moreover,
$\langle Bv,q \rangle = (v,\gamma q)_{L^2(\Gamma_{N})}$ and
$\langle \ell,q \rangle = (f,q)_{L^2(\Omega)} - (g,\gamma q)_{L^2(\Gamma_{N})}$,
where $f \in L^2(\Omega)$ and $g \in L^2(\Gamma_{N})$.

The cost functional (\ref{eq:cost}) is
\begin{equation*}
J(v) := \| \nabla (y(v) - y^d) \|_{L^2(\Omega)}^2
+ \alpha \|  v - u^d \|_{L^2(\Gamma_{N})}^2 ,
\end{equation*}
and the state equation (\ref{eq:state:var}) is
\begin{equation*} 
(\nabla y(v), \nabla q)_{L^2(\Omega,\mathbb{R}^d)} =
(f,q)_{L^2(\Omega)} + (g+v,\gamma q)_{L^2(\Gamma_{N})} ,
\quad \forall q \in V_0 .
\end{equation*}
It has the classical form
\[
\left\{
\begin{array}{rclr}
- \Delta y(v) & = & f    \quad & \textrm{a.e. in } \Omega, \\
    y(v)      & = & 0         \quad & \textrm{on } \Gamma_D, \\
    \tfrac{\partial y(v)}{\partial n} & = & g + v \quad & \textrm{on } \Gamma_{N} .
\end{array}
\right.
\]
The majorant (\ref{eq:major}) has the form (see, e.g., \cite[Sect. 4.1]{Repin2008}
for details)
\begin{multline*} 
\maj^2(q,\tau,\beta) =
(1+\beta) \| \tau - \nabla q \|_{L^2(\Omega,\mathbb{R}^d)}^2  \\
+
\frac{1+\beta}{\beta} \left( c_{\Omega,2}^2 \| \dvg \tau + f \|_{L^2(\Omega)}^2 +
c_{\Gamma_N}^2 \| \tfrac{\partial \tau}{\partial n} + g + v \|_{L^2(\Gamma_N)}^2 \right),
\end{multline*}
where constants satisfy
\[
\| q \|_{L^2(\Omega)} \leq c_{\Omega,2} \| \nabla q \|_{L^2(\Omega,\mathbb{R}^2)}
\quad \textrm{ and} \quad
\| q \|_{L^2(\Gamma_N)} \leq c_{\Gamma_N} \| \nabla q \|_{L^2(\Omega,\mathbb{R}^2)} ,
\quad \forall q \in V_0 .
\]
\begin{Proposition} \label{pr:ex2:Jupmin}
For all $q \in \H1o$, $\tau \in \Hdiv$, and $\beta > 0$
\begin{align*}
\costup(q,\hat \tau,\beta) & = \inf\limits_{\tau \in \Hdiv} \costup(q,\tau,\beta) , \\
\costup(q,\tau,\hat \beta) & = \inf\limits_{\beta > 0} \costup(q,\tau,\beta) ,
\end{align*}
where $\hat \tau$ satisfies
\begin{multline*} 
\beta (\hat \tau, \xi)_{L^2(\Omega,\mathbb{R}^d)} +
c_\Omega^2 (\dvg \hat \tau, \dvg \xi )_{L^2(\Omega)} +
c_{\Gamma_N}^2 (\tfrac{\partial \hat \tau}{\partial n},\tfrac{\partial \xi}{\partial n})_{L^2(\Gamma_{N})}
\\
=
\beta (\nabla q, \xi )_{L^2(\Omega,\mathbb{R}^d)} +
c_\Omega^2 (f+v, \dvg \xi )_{L^2(\Omega)} +
c_{\Gamma_N}^2 (g+v,\tfrac{\partial \xi}{\partial n})_{L^2(\Gamma_{N})}
,
\quad \forall \xi \in \Hdiv
\end{multline*}
and
\begin{equation*} 
\hat \beta = \frac{\left( c_{\Omega,2}^2 \| \dvg \tau + f \|_{L^2(\Omega)}^2 +
c_{\Gamma_N}^2 \| \tfrac{\partial \tau}{\partial n} + g + v \|_{L^2(\Gamma_N)}^2 \right)^{1/2}}
{\| \tau - \nabla q \|_{L^2(\Omega,\mathbb{R}^d)}} .
\end{equation*}
\end{Proposition}

\subsection{Algorithms}
\label{se:alg}

The results of Sect. \ref{se:est} give grounds for several
error estimation Algorithms. Note that the estimates in Theorems \ref{th:costest} and
\ref{th:cont:est} are valid for any
approximations from $\Cspace_{\rm ad}$. There is no need for Galerkin
orthogonality, extra regularity, or mesh dependent data. Thus they can
be combined with any existing numerical scheme, which generates approximations
of the optimal control (and/or state).
Computation of the derived estimates requires
some finite dimensional subspaces. Hereafter, assume that
$\Cspace_{\rm ad}^h\subset \Cspace_{\rm ad}$
$\Sspace_0^h \subset \Sspace_0$ and $\Fspace^h \subset \Fspace$ are given.
They can be generated, e.g., by finite elements or Fourier
series. The approximate solution of (\ref{eq:state:weak}) is
$y^h(v) \in \Sspace_0^h \subset \Sspace_0$ that satisfies
\begin{equation} \label{eq:state:num}
a(y^h(v),z) = \langle Bv + \ell , z \rangle_{\Sspace_0}, \quad \forall z \in \Sspace_0^h .
\end{equation}
\begin{Remark}
By Remark \ref{re:boundinfo}, the evaluation of
(the approximation of) $J(v)$ by computing $y^h(v)$ from (\ref{eq:state:num})
and
$J_h(v) := \NT y^h(v) - y^d \NT^2 + \| v- u^d \|_{\mathcal N}^2$ coincides with the
lower bound
$\costlow(v,y^h(v)) = \max\limits_{y \in \Sspace_0^h} \costlow(v,y)$.
\end{Remark}
The generation of the estimates for the cost function value $J(v)$ for a given
approximation $v \in \Cspace_{\rm ad}$ is depicted as Algorithm \ref{alg:costest}.
\begin{algorithm}[tb]
\caption{Generation of bounds for the cost functional value}
\label{alg:costest}
\begin{algorithmic}
\STATE {\bf input:}
$v \in \Cspace_{\rm ad}$ \COMMENT{approximation of the control}
$\Sspace_0^h$ \COMMENT{subspace for state},
$\Fspace^h$ \COMMENT{subspace for the flux of state},
$I_{\max}$ \COMMENT{maximum number of iterations},
$\varepsilon$ \COMMENT{stopping criteria}

\STATE {}

\STATE $y^h = \argmaxl_{y \in \Sspace_0^h} \costlow(v,y)$
\COMMENT{compute $y^h(v)$ from (\ref{eq:state:num})}

\STATE $\hat v^h = \argminl_{v \in \Cspace_{\rm ad}} \costlow(v,y^h)$
\COMMENT{compute $\hat v(y^h)$ by (\ref{eq:vhatdef})}

\STATE {}

\STATE $\beta^0 = 1$
\FOR{$\; k = 1 \;$ {\bf to} $\; I_{\max} \;$}

    \STATE $\tau^k = \argminl_{\tau \in \Fspace^h} \costup(v,\tau,\beta^{k-1})$
    \STATE $\beta^k = \argminl_{\beta > 0} \costup(v,\tau^k,\beta)$

    \IF{
    $
    \tfrac{\costup(v,\tau^{k-1},\beta^{k-1})-\costup(v,\tau^{k},\beta^{k})}
    {\costup(v,\tau^{k},\beta^{k})} < \varepsilon
    $
    }
    \STATE{\bf break}
    \ENDIF

\ENDFOR
\STATE $\costlow_h(v) = \costlow(v,y^h)$
\STATE $\costup_h(v) = \costup(v,\tau^k,\beta^k)$
\STATE $\costlow_h(u) = \costlow(\hat v^h,y^h)$

\STATE {}

\STATE {\bf output:}
$\costlow_h(u)$ \COMMENT{lower bound for $J(v)$},
$\costup_h(v)$ \COMMENT{upper bound for $J(v)$},
$\costlow_h(u)$ \COMMENT{lower bound for $J(u)$},
%
\end{algorithmic}
\end{algorithm}

In order to test the presented error estimates, a projected gradient method
(see, e.g., \cite{GruverSachs1981,Kelley1999})
is applied to generate a sequence approximations. Method consists of line
searches along (anti)gradient directions, where all evaluated points are first projected
to the admissible set.
A projected gradient method with error estimates is depicted as Algorithm
\ref{alg:projgrad}.
\begin{algorithm}[tb]
\caption{Projected gradient method with guaranteed cost estimates}
\label{alg:projgrad}
\begin{algorithmic}
\STATE {\bf input:}
$v^0 \in \Cspace_{\rm ad}$ \COMMENT{initial approximation of the control}
$\Sspace_0^h$ \COMMENT{subspace for state},
$\Fspace^h$ \COMMENT{subspace for the flux of state},
$I_{\max}^{PG}$ \COMMENT{maximum number of iterations (projected gradient)},
$\varepsilon^{PG}$ \COMMENT{stopping criteria (projected gradient)}
$I_{\max}$ \COMMENT{maximum number of iterations ($\costup$ minimization)},
$\varepsilon$ \COMMENT{stopping criteria ($\costup$ minimization)}

\STATE {}

\FOR{$\; k = 0 \;$ {\bf to} $\; I_{\max}^{\rm PG}$}

    \STATE
    $
    \left\{ \costlow_h(v^k), \, \costup_h(v^k), \, \costlow_h^k(u) \right\} =
    \textrm{GenerateCostEstimates}(v^k,\Sspace_0^h,\Fspace^h,I_{\max},\varepsilon)
    $

    \STATE
    $
    d^k = 2 \left( B^*(y^d - y(v^k)) + \mathcal N (u^d - v^k) \right)
    $
    \COMMENT{search direction}

    \STATE
    $
    s^k \argmin\limits_{\lambda \in [0,\lambda_{\max}]}
    J_h \left( \Pi_{\rm ad} \left( v^k + \lambda d(v^k) \right) \right)
    $
    \COMMENT{step length (golden section method)}

    \STATE
    $
    v^{k+1} = v^k + s^k d^k
    $
    \COMMENT{update approximation}

    \IF{
    $
    \tfrac{\| v^k - v^{k-1}\|}
    {\| v^{k-1} \|} < \varepsilon^{\rm PG}
    $
    }
    \STATE{\bf break}
    \ENDIF

\ENDFOR

\STATE {}

\STATE {\bf output:}
$\left\{ v^k) \right\}_{k=1}^N$ \COMMENT{sequence of approximations},
$\left\{ \costlow_h(v^k) \right\}_{k=1}^N$ \COMMENT{lower bounds for $J(v^n)$},
$\left\{ \costup_h(v^k) \right\}_{k=1}^N$ \COMMENT{upper bounds for $J(v^n)$},
$\costlow_h^N(u)$ \COMMENT{lower bound for $J(u)$}
\end{algorithmic}
\end{algorithm}
At the beginning of every projected gradient step Algorithm \ref{alg:costest}
is used to generate approximations for the cost functional. After the execution
of Algorithm \ref{alg:projgrad} ($N$ iteration steps taken), cost estimates
are recalled to generate two-sided estimates for $\err(v)$ (i.e., the
difference $J(v)-J(u)$) at each iteration step
($k=1,\dots,N$) as follows:
\begin{align*}
\err^2(v^k) & \geq \costlow^2_h(v^k) - \costup_h(v^N)  \\
\err^2(v^k) & \leq \costup_h(v^k)  - \costlow_h^N(u)
\end{align*}
Note that the iterate of the last step ($N$'th step) is used to generate
as accurate bounds as possible for $J(u)$.

\subsection{Numerical tests}
\label{se:num}

Finite dimensional subspaces are generated by the finite element method
(see, e.g., \cite{Ciarlet1978}). In these tests, $\Cspace=L^2(\Omega)$, $\Sspace_0=\H1o$,
and $\Fspace=\Hdiv$. Subspaces $DG_h^p \subset L^2(\Omega)$, $V_h^p \subset \H1o$,
and ${\rm RT}^p \subset \Hdiv$
are generated by Discontinous Galerkin elements, Lagrange elements, and
Raviart-Thomas elements, respectively. Superscripts $p$ denote the order of basis functions.
All the numerical tests were performed using FEniCS (see \cite[Ch. 3]{LoggMardalEtAl2012a}
for detailed descriptions of the applied elements and for additional references).
\begin{Example} \label{ex:dir:co:loc}
Let $\Omega = (0,1)^2$.
Consider the optimal control problem generated by (\ref{eq:ex1:cost}),
(\ref{eq:ex1:state}), and $\Cspace_{\rm ad}$ defined by
(\ref{eq:ex1:Uad:loc}), where $\psi_{-}(x_1,x_2) = -3$ and
$\psi_{+}(x_1,x_2) = 3$. Select
\begin{align*}
y(x_1,x_2) & = \sin(k_1 \pi x_1) \sin(k_1 \pi x_2), \\
y^d(x_1,x_2) & = \sin(k_1 \pi x_1) \sin(k_1 \pi x_2)
+ \beta \sin(m_1 \pi x_1) \sin(m_1 \pi x_2) , \\
u^d(x_1,x_2) & = 0 \\
u(x_1,x_2) & =
\max\left\{ \psi_{-}(x_1,x_2) ,
\min\left\{ \psi_{+}(x_1,x_2) , \tfrac{\beta}{\alpha} \sin(m_1 \pi x_1) \sin(m_1 \pi x_2) \right\} \right\} \\
f(x_1,x_2) & = \pi^2 (k_1^2+k_2^2) \sin(k_1 \pi x_1) \sin(k_1 \pi x_2) - u(x_1,x_2) ,
\end{align*}
where $k_1,k_2,m_1,m_2 \in \mathbb{Z}$ and $\beta \in \mathbb{R}$.
\end{Example}
In Example \ref{ex:dir:co:loc}, select $k_1=1$, $k_2=1$, $m_1=2$, $m_2=1$,
$\beta = 0.5$, and $\alpha = 0.05$. A mesh of 50$\times$50 cells divided to triangular
elements is being used. Consider first linear elements, i.e., $p_1=p_2=p_3=1$, the
amount of corresponding global degrees of freedom are ${\rm dim}({\rm DG}_h^1) = 15000$,
${\rm dim}({V}_h^1) = 2601$, and ${\rm dim}({\rm RT}_h^1) = 7600$.
The bounds generated by
Algorithm \ref{alg:projgrad} ($I_{\rm max}^{PG}=10$) are depicted in
Figure \ref{fig:ex:U1V1Q1}.
\begin{figure}
\begin{center}
\begin{tabular}{l}
\includegraphics[scale=0.5]{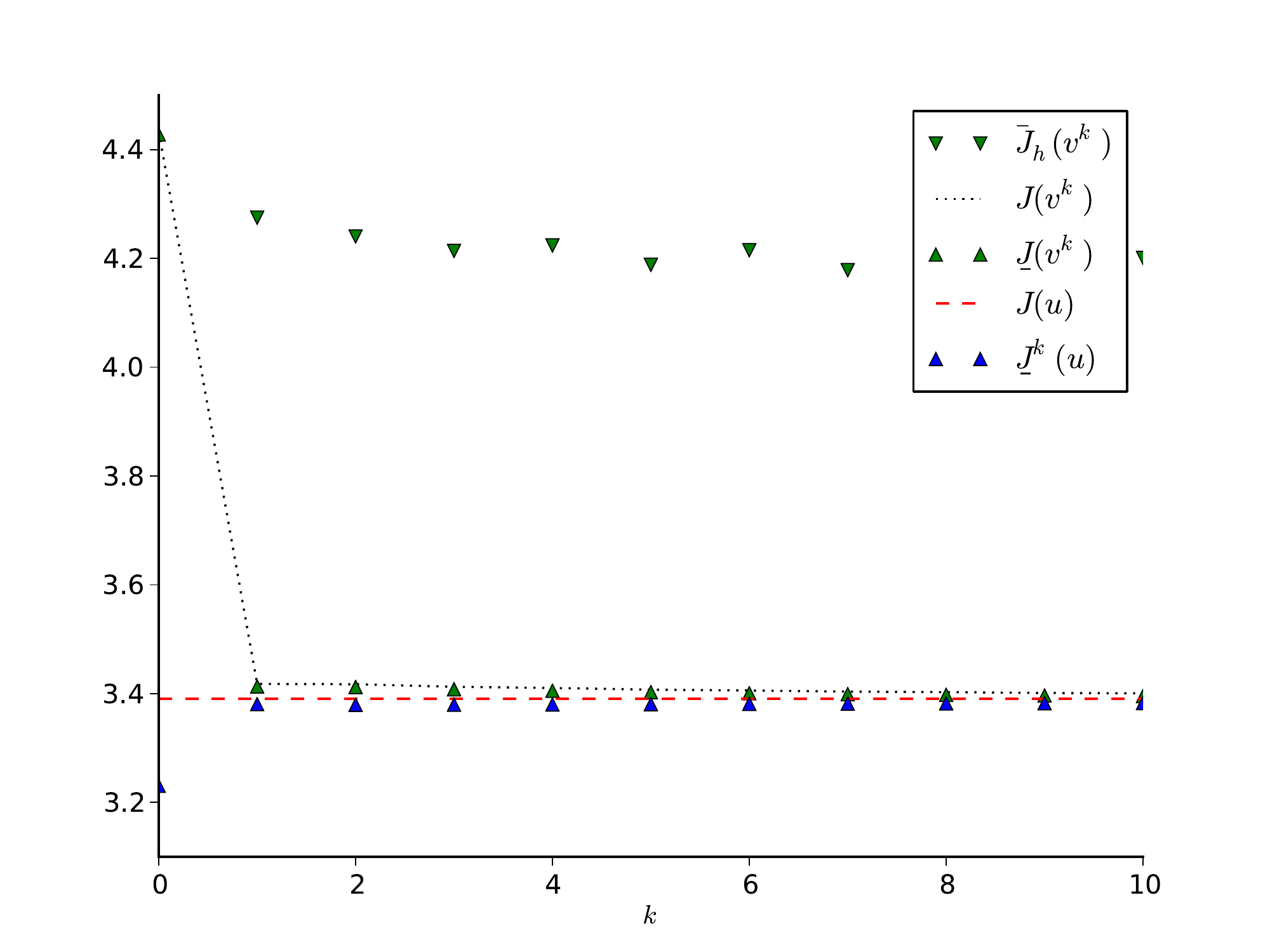} \\
\includegraphics[scale=0.5]{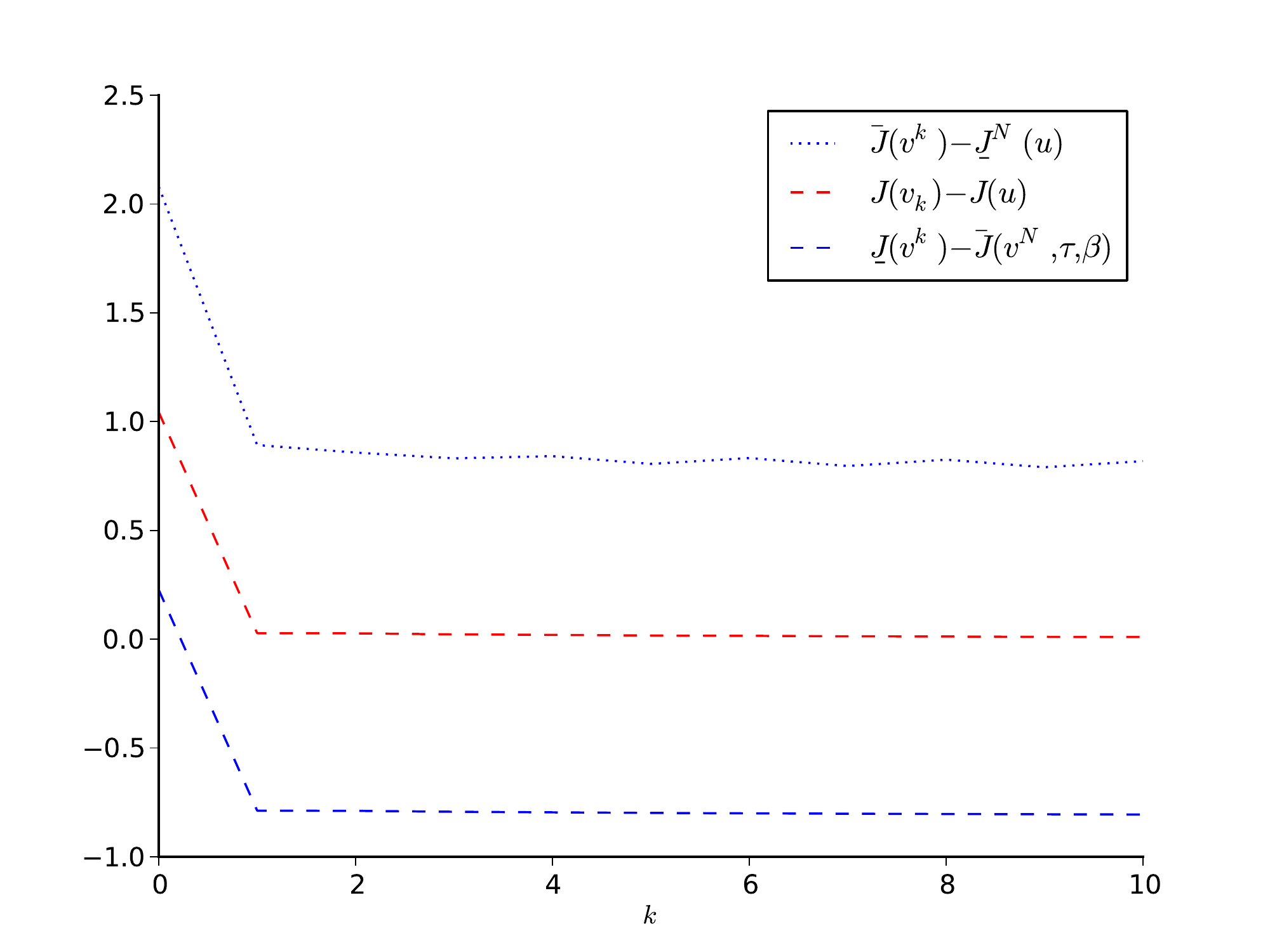}
\end{tabular}
\caption{Estimates for the cost function value (top) and the error quantity (bottom),
where subspaces for control, state, and flux are ${\rm DG}_h^1$, ${V}_h^1$, and
${\rm RT}_h^1$, respectively.}
\label{fig:ex:U1V1Q1}
\end{center}
\end{figure}
If the order of approximation for state and flux are increased, i.e., subspaces $\Sspace_h$
and $\Fspace_h$ are enhanced, then the accuracy of error bounds improves significantly
(see Fig. \ref{fig:ex:U1V2Q2}). Here ${\rm dim}({V}_h^2) = 10201$
and ${\rm dim}({\rm RT}_h^2) = 25200$
\begin{figure}
\begin{center}
\begin{tabular}{l}
\includegraphics[scale=0.5]{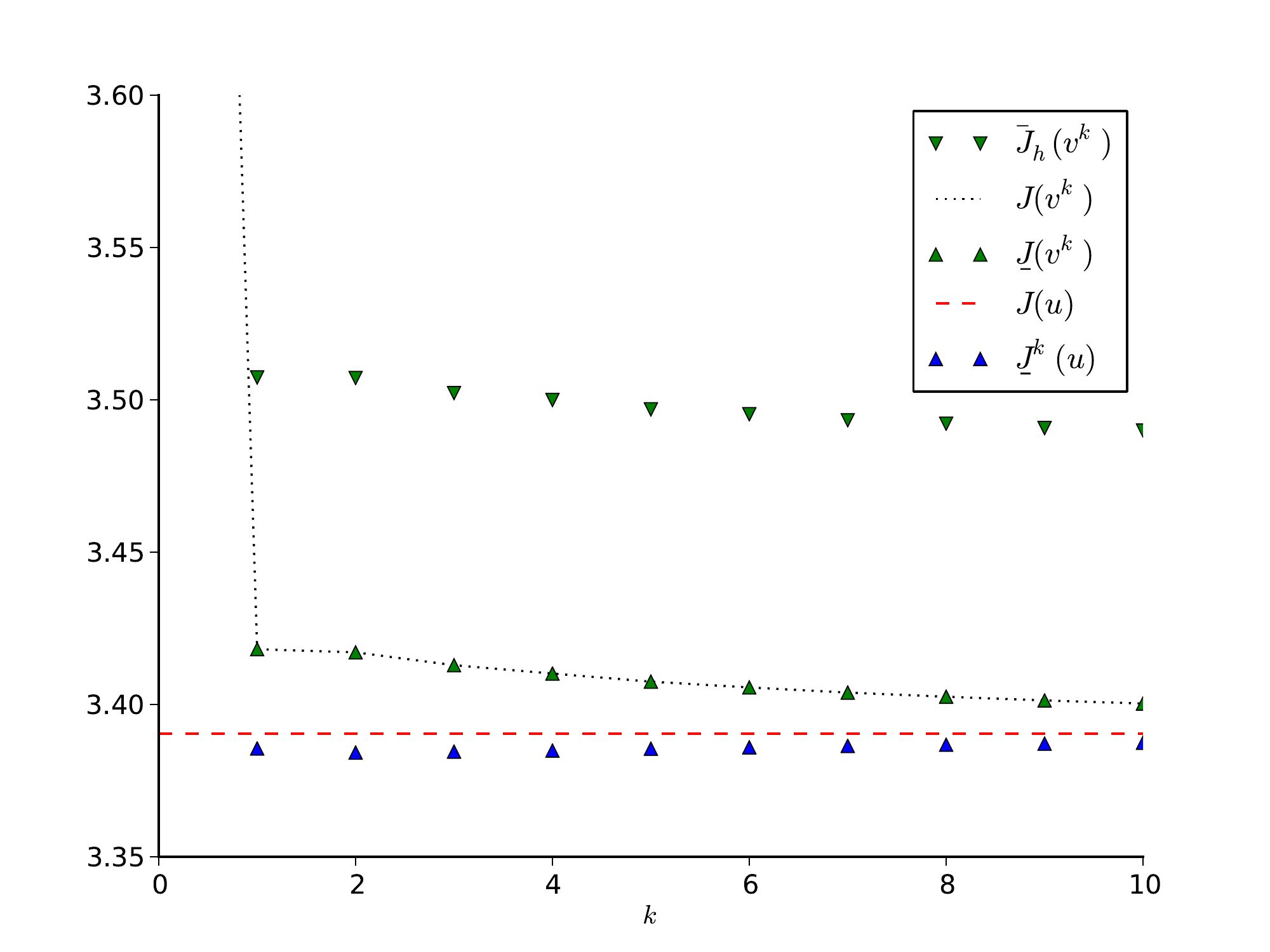} \\
\includegraphics[scale=0.5]{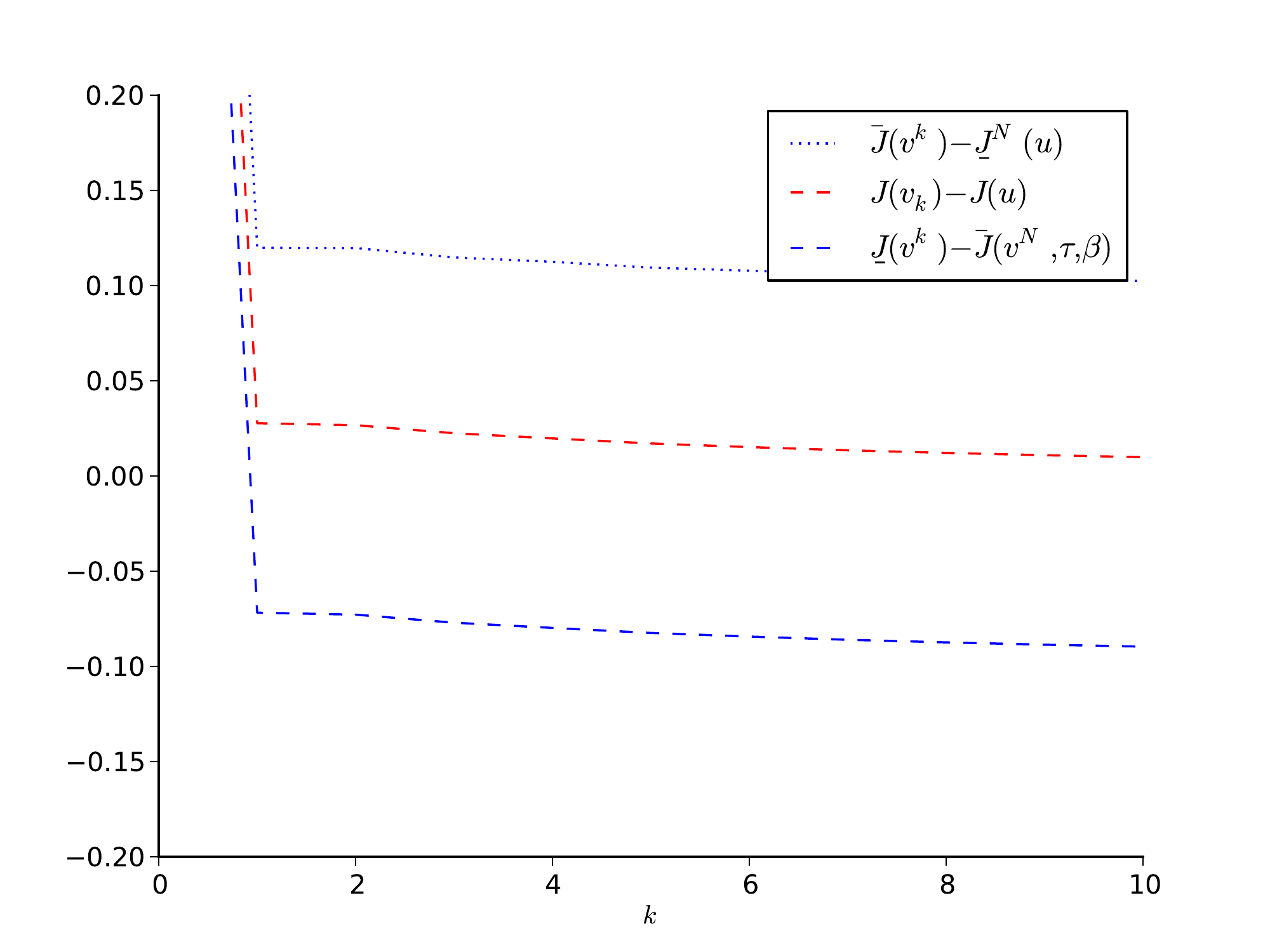}
\end{tabular}
\caption{Estimates for the cost function value (top) and the error quantity (bottom),
where subspaces for control, state, and flux are ${\rm DG}_h^1$, ${V}_h^2$, and
${\rm RT}_h^2$, respectively.}
\label{fig:ex:U1V2Q2}
\end{center}
\end{figure}
In previous examples, $J(v)$ and $J(u)$ (and other integrals also) were computed using a
uniformly refined mesh and 121 integration points in each triangle.

Obviously, the negative lower bound for the error could be rejected immediately. Sharp
lower bound requires a very good approximation of the optimal control $v \approx u$ and the
corresponding flux of the respective state $\tau \approx \nabla y(u)$. Then the upper
bound $J(u) \leq J(v) \leq \costup(v,\tau,\beta)$ would be very efficient. However, ten steps
of the projected gradient method does not provide a very accurate approximation.
It is a matter of further numerical tests to apply more efficient approximation methods
(see, e.g., \cite{ItoKunisch2008}) and to apply the element wise contributions of the
error estimates to generate adaptive sequences of subspaces. 


\begin{thebibliography}{10}

\bibitem{Ciarlet1978}
P.~G. Ciarlet.
\newblock {\em The finite element method for elliptic problems}.
\newblock North-Holland Publishing Co., Amsterdam, 1978.
\newblock Studies in Mathematics and its Applications, Vol. 4.

\bibitem{Clarke2013}
F.~Clarke.
\newblock {\em Functional analysis, calculus of variations and optimal
  control}, volume 264 of {\em Graduate Texts in Mathematics}.
\newblock Springer, London, 2013.

\bibitem{EkelandTemam1976}
I.~Ekeland and R.~Temam.
\newblock {\em Convex Analysis and Variational Problems}.
\newblock North--Holland, New York, 1976.

\bibitem{GaevskayaHoppeRepin2007}
A.~Gaevskaya, R.~W.~H. Hoppe, and S.~Repin.
\newblock A posteriori error estimation for elliptic optimal control problems
  with distributed control.
\newblock {\em J. Math. Sci. (N. Y.)}, 144:4535--4547, 2007.

\bibitem{GruverSachs1981}
W.~A. Gruver and E.~Sachs.
\newblock {\em Algorithmic methods in optimal control}, volume~47 of {\em
  Research Notes in Mathematics}.
\newblock Pitman (Advanced Publishing Program), Boston, Mass.-London, 1981.

\bibitem{ItoKunisch2008}
K.~Ito and K.~Kunisch.
\newblock {\em Lagrange multiplier approach to variational problems and
  applications}, volume~15 of {\em Advances in Design and Control}.
\newblock Society for Industrial and Applied Mathematics (SIAM), Philadelphia,
  PA, 2008.

\bibitem{Kelley1999}
C.~T. Kelley.
\newblock {\em Iterative methods for optimization}, volume~18 of {\em Frontiers
  in Applied Mathematics}.
\newblock Society for Industrial and Applied Mathematics (SIAM), Philadelphia,
  PA, 1999.

\bibitem{Lions1971}
J.-L. Lions.
\newblock {\em Optimal control of systems governed by partial differential
  equations.}
\newblock Translated from the French by S. K. Mitter. Die Grundlehren der
  mathematischen Wissenschaften, Band 170. Springer-Verlag, New York, 1971.

\bibitem{LoggMardalEtAl2012a}
A.~Logg, K.-A. Mardal, G.~N. Wells, et~al.
\newblock {\em Automated Solution of Differential Equations by the Finite
  Element Method}.
\newblock Springer, 2012.

\bibitem{MaliNeittaanmakiRepin2014}
O.~Mali, S.~Repin, and P.~Neittaanm{\"a}ki.
\newblock {\em Accuracy verification methods, theory and algorithms}, volume~32
  of {\em Computational Methods in Applied Sciences}.
\newblock Springer, 2014.

\bibitem{Mikhlin1964}
S.~G. Mikhlin.
\newblock {\em Variational methods in mathematical physics}.
\newblock Translated by T. Boddington; editorial introduction by L. I. G.
  Chambers. A Pergamon Press Book. The Macmillan Co., New York, 1964.

\bibitem{Moreau1965}
J.-J. Moreau.
\newblock Proximit\'e et dualit\'e dans un espace hilbertien.
\newblock {\em Bull. Soc. Math. France}, 93:273--299, 1965.

\bibitem{NeittaanmakiRepin2004}
P.~Neittaanm{\"a}ki and S.~Repin.
\newblock {\em Reliable methods for computer simulation, Error control and a
  posteriori estimates}.
\newblock Elsevier, New York, 2004.

\bibitem{Repin1997}
S.~Repin.
\newblock A posteriori estimates for approximate solutions of variational
  problems with strongly convex functionals.
\newblock {\em Problems of Mathematical Analysis}, 17:199--226, 1997.

\bibitem{Repin2000}
S.~Repin.
\newblock A posteriori error estimation for variational problems with uniformly
  convex functionals.
\newblock {\em Math. Comp.}, 69(230):481--500, 2000.

\bibitem{Repin2008}
S.~Repin.
\newblock {\em A posteriori estimates for partial differential equations},
  volume~4 of {\em Radon Series on Computational and Applied Mathematics}.
\newblock Walter de Gruyter GmbH \& Co. KG, Berlin, 2008.

\bibitem{Troltzsch2010}
F.~Tr{\"o}ltzsch.
\newblock {\em Optimal control of partial differential equations}, volume 112
  of {\em Graduate Studies in Mathematics}.
\newblock American Mathematical Society, Providence, RI, 2010.
\newblock Theory, methods and applications, Translated from the 2005 German
  original by J{\"u}rgen Sprekels.

\end{thebibliography}

\def\cprime{$'$}

\end{document}